%% file: sar-final.tex
\documentclass{article}

\include{standard}

\title{Cancellation of Singularities in SAR for \\ Curved Flight Paths and Non-Flat Topography}
\author{Andrew Homan}
\date{}

\begin{document}
\maketitle

\begin{abstract}
  We consider a mathematical model of synthetic aperture radar (SAR)
  with a known, possibly non-flat, topography. In this context we
  consider the problem of recovering the wavefront set of the ground
  reflectivity, given radar data measured along a curved flight
  path. We show that if singularities are located at ``mirror
  points,'' then the resulting data may be smooth; in effect, the
  singularities ``cancel.'' With a flat topography, these mirror
  points are always discrete, but we show that in a non-flat
  topography there may be infinite families of mirror points.
\end{abstract}

\section{Introduction}

Synthetic aperture radar (SAR) is a radar imaging technique using an
airplane or satellite flying along a known flight path to illuminate
the Earth (or another planet) with radar, while at the same time
measuring the backscattered signal. From this data, one would like to
image the electromagnetic reflectivity of the ground. We refer the
reader to \cite{cheney09} and references therein for a practical
introduction to SAR. It has been adapted for many applications,
including imaging through the forest canopy \cite{cheney04, cloude04,
  varslot10}, imaging ocean waves \cite{collard05, swift79}, and more
recently for material identification \cite{albanese13}.

When the same antenna is used to produce the incident wave and measure
the scattered wave, the configuration is referred to as monostatic
SAR, which is the object of our study. Under an assumption that the
surface being imaged is approximately flat, the forward operator
mapping scene to data can be modelled as a restricted circular Radon
transform \cite{hellsten87, stefanov13}. This operator has been
studied in the context of integral geometry \cite{ambartsoumian05,
  ambartsoumian06, agranovsky07} and also as a model of thermoacoustic
tomography \cite{haltmeier07, finch04, kuchment08}.

A related SAR configuration requires two antennas, one emitting and
one receiving, which follow distinct flight paths. This is called
bistatic SAR, and there has been some study of the restricted
elliptical Radon transform as a model for it \cite{krishnan11,
  coker07}. It can also be used to model the artifacts occuring in
monostatic SAR in the presence of a reflecting wall
\cite{ambartsoumian12}. While we do not study bistatic SAR here, these
works are notable for their use of microlocal analysis and the
interpretation of the forward operator as a Fourier integral
operator. This is similar to our approach here.

The problem of object determination in SAR lends itself naturally to
microlocal analysis. For applications in target identification, it is
often only necessary to recover the outline of the object, which tends
to present itself as a jump discontinuity in the reconstruction
\cite{cheney03}. Under the flat surface approximation, Nolan and
Cheney showed the forward operator is a Fourier integeral operator
\cite{nolan04}. In \cite{nolan03}, they also give an example of SAR
with non-flat topography and mention some of its limitations. In this
paper we elaborate on these limitations, particularly in the case of
non-flat topography.

Even in a flat model of SAR, artifacts can appear. These artifacts are
explained by the fact that the canonical relation is two-to-one in a
neighborhood of the image of these artifacts. If, in addition, the
flight path is straight and at a constant altitude, there is a
well-known left-right ambiguity that obstructs the unique
reconstruction of the scene on both sides of the flight
path. Previously, using the restricted circular Radon transform as a
model, Stefanov and Uhlmann showed that this left-right ambiguity
persists in a microlocal sense when the flight path is curved
\cite{stefanov13}.

Their analysis is built upon a notion of ``mirror points'' which
characterize the artifacts in a way that is determined only by the
geometry of the surface and the flight path. They show that the
artifacts that appear at one mirror point are unitary images of the
real signal under a microlocal Fourier integral operator. In this way
each pair of mirror points induces an infinite family of singular
reflectivity functions whose image under the forward operator is
smooth. This phenomena, called ``cancellation of singularities,''
presents an obstruction to any unique and stable reconstruction of the
scene.

Our main goal is to characterize the artifacts in SAR with an
arbitrary topography and arbitrary flight path. This requires
analyzing the forward operator as a Fourier integral operator, which
we do in the second section. Examining its canonical relation leads us
to a generalization of the notion of a mirror points. In contrast to
the case of flat topography, we find that an infinite number of mirror
points may be related to a single observed singularity. Finally, we
show cancellation of singularities for pairs of isolated mirror
points, and present an example showing cancellation of singularities
for an infinite family of mirror points.

\section{Preliminaries}

\subsection{Model}

In this section, we summarize the standard model of SAR \cite{nolan03,
  nolan04}. Each component of the electromagnetic field satisfies the
wave equation $\partial_t^2u - c^2\Delta u = 0$. We assume the
following data are known:
\begin{enumerate}
\item The flight path of the aircraft, given by a smooth, embedded
  curve $\gamma$ parameterized with unit speed.
\item The surface of the Earth, $\Psi$, given by a smooth, embedded
  surface $\psi$.
\end{enumerate}
We will use $s$ to parameterize the flight path, and use $(u, v)$ to
locally parameterize $\Psi$. Let $c$ be the speed of electromagnetic
propagation; it is commonly assumed that $c$ is the speed of light in
air $c_0$ except for a singular perturbation supported on $\Psi$ of
the form:
\[
c_0^{-2} - c^{-2} = V(u, v)\delta(\psi(u, v) - (x, y, z)).
\]
Here, $c_0$ is the constant background speed of propagation and $V$ is
the ground reflectivity function, which we seek to recover from the
electromagnetic field measured on the antenna as it travels along the
flight path. We will assume that data is recorded over a finite
interval $\mathcal{Y} = (s_1, s_2) \times (t_1, t_2)$. Under some
assumptions on the geometry of the antenna, \cite{nolan03} finds the
map from $V$ to the recorded data is of the form,
\[
 FV(s, t) = \int_{\R \times \mathcal{X}} A(u, v, s, t,
 \omega)e^{-i\omega\left(t - \frac{2}{c_0}|\psi(u, v) -
   \gamma(s)|\right)}V(u, v) \, du\, dv\, d\omega,
\]
where $A$ is an amplitude of order two. By finite speed of
propagation, the amplitude $A$ is zero outside the ``visible'' set
\[
 \mathcal{X} = \left\{ (u, v) : \exists s \in (s_1, s_2),
 \frac{2}{c_0}|\psi(u, v) - \gamma(s)| \in (t_1, t_2) \right\}.
\]
In particular, this implies that $A$ has proper support.

In what follows we will assume that the wavefront set of $V$ is
contained in $T^*\mathcal{X}\setminus 0$, i.e., $V \in \cinf(\Psi
\setminus \overline{\mathcal{X}})$. This is not to imply that all
singularities in $T^*\mathcal{X}\setminus 0$ are recoverable from
$FV$; only that ``invisible'' ones cannot be.

In order to recover $V$ from $FV$, our goal will be to determine the
circumstances under which $F$ is a Fourier integral operator, and in
those cases apply microlocal techniques to recover the wavefront set
of $V$ uniquely. In some cases, the geometry of $\Psi$ and $\gamma$
prevent the unique recovery of singularities, and we provide for all
choices of $\Psi$ and $\gamma$ an infinite family of counterexamples
$V$ with nonempty wavefront set whose image, $FV$, is smooth --
therefore, the singularities of $V$ cannot be recovered.

We conclude this section with some convenient notation for the
geometry of the flight path, adopted from \cite{nolan04}.

\begin{notation}
We will write $R(u, v, s) = \psi(u, v) - \gamma(s)$. Then the time
required for a signal to leave $\gamma(s)$, reflect off $\psi(u, v)$,
and return is $2|R(u, v, s)|/c_0$. The reflection off $\psi(u, v)$
propagates in the direction $\hat{R}(u, v, s) = R(u, v, s)/|R(u, v,
s)|$.

In what follows, we often use the notation shown in Figure
\ref{notation}. Let the projection of the vector $\hat{R}(u, v, s)$
onto the tangent plane of $\Psi$ at $(u, v)$ be $\pi_{T\Psi}\hat{R}(u,
v, s)$.

\begin{figure}[h]
  \begin{center}
    \includegraphics[width=3in]{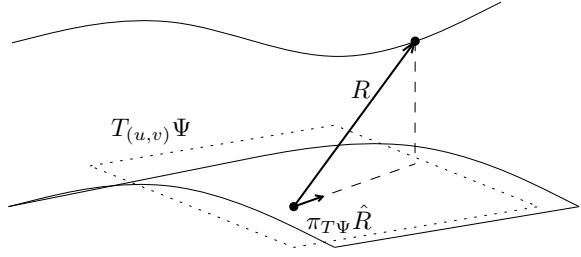}
    \caption{Notation for geometry of flight path and topography.}
    \label{notation}
  \end{center}
\end{figure}
\end{notation}

\subsection{Microlocal analysis of the forward operator}

The model described in the previous section takes the forward operator
to be an oscillating integral with phase
\[
 \phi(s, t, u, v, \omega) = \omega\left(t - \frac{2}{c_0}|R(u, v, s)|\right).
\]
We take $(s, t, \sigma, \tau)$ as local coordinates on
$T^*\mathcal{Y}$ and $(u, v, \xi, \eta)$ as local coordinates on
$T^*\mathcal{X}$.

\begin{proposition}
If $\mathrm{dist}(\Psi, \gamma) > 0$, then $F$ is a Fourier integral
operator associated to the Lagrangian submanifold
\begin{align*}
 \Lambda = \left\{ (s, t, \sigma, \tau; u, v, \xi, \eta) : \phantom{\frac{1}{2}}\right. 
&t = \frac{2}{c_0}|R(u, v, s)|, \\
& \sigma = \frac{2\tau}{c_0}\hat{R}(u, v, s)\cdot\dot{\gamma}(s), \\
& \left. (\xi, \eta) = -\frac{2\tau}{c_0}\pi_{T\Psi}\hat{R}(u, v, s) \right\}.
\end{align*}
\end{proposition}

\begin{proof}
The critical points of the phase function (with respect to $\omega$)
lie on the set
\[
 C = \left\{ t = \frac{2}{c_0}|R(u, v, s)| \right\}.
\]
The function $t - 2|R(u, v, s)|/c_0$ is smooth away from the set $\{
|R(u, v, s)| = 0 \}$. By our assumption that the flight path is
separated from the surface of the Earth, this poses no problem. By the
inverse theorem, $C$ is a submanifold of $\mathcal{Y}\times\mathcal{X}
\times (\R\setminus 0)$. Define $\Lambda$ to be the image of $C$ under
the map:
\[
 (s, t, u, v, \omega) \mapsto (s, t, \partial_s\phi, \partial_t\phi, u, v, \partial_u\phi, \partial_v\phi).
\]
By \cite[Lemma 2.3.2]{duistermaat96}, this is an immersed Lagrangian
submanifold of $(T^*\mathcal{Y} \times T^*\mathcal{X}) \setminus 0$
with respect to $\omega_{\mathcal{Y}} \oplus \omega_{\mathcal{X}}$,
where $\omega_{\mathcal{X}}$ and $\omega_{\mathcal{Y}}$ are the
canonical symplectic forms on $T^*\mathcal{X} \setminus 0$ and
$T^*\mathcal{Y} \setminus 0$, respectively.
\end{proof}

The canonical relation $\Lambda^\prime$ of $F$ is found by multiplying
the last two coordinates $(\xi, \eta)$ by $-1$. Recall that if $V \in
\cs(X)$, then
\[
 \WF(FV) \subset \{ q \in T^*\mathcal{Y}\setminus 0 : (p, q) \in \Lambda^\prime, p \in \WF(V) \}.
\]

As the data we can measure is confined to a finite interval $(s_1,
s_2), (t_1, t_2)$, the source of any singularity in the data is
confined to a subset of $T^*\mathcal{X} \setminus 0$ whose projection
to $\mathcal{X}$ is bounded. To see this, note that the canonical
relation requires the travel-time from the source of the signal to the
plane to be bounded by $t_2$. This also can be seen from finite speed
of propagation.

\subsection{Sets on which $\Lambda^\prime$ is locally of graph type}

Interpreting the forward operator $F$ as a Fourier integral operator
allows us to exploit the geometry of its canonical relation. In
particular, there are some microlocalizations of $F$ that are elliptic
in the sense of \cite[Def 25.3.4]{horm85}; that is, their canonical
relation is the graph of a symplectomorphism and their amplitude
(relative to the phase $\phi$ in the usual representation) is nonzero
in a neighborhood of their micro-support. 

Here we determine the largest set on which $\Lambda^\prime$ is locally
of graph type, with a view toward constructing microlocal parametrices
of $F$ later. At this point we distinguish between a canonical
relation which is the graph of a diffeomorphism and a canonical
relation which is the graph of a bijective diffeomorphism. However,
there is no distinction locally, as we can always restrict to a yet
smaller neighborhood on which the diffeomorphism is bijective.

\begin{proposition}
$\Lambda^\prime$ is a homogeneous canonical relation that is locally
  of graph type away from the degenerate set $\Sigma = \Sigma_1 \cup
  \Sigma_2$, where
\[
 \Sigma_1 = \left\{ (s, t, \sigma, \tau, u, v, \xi, \eta) : \pi_{T\Psi}\hat{R}(u, v, s) \parallel \nabla_{u, v}(\hat{R}(u, v, s) \cdot
 \dot{\gamma}(s)) \right\} \cap \Lambda^\prime,
\]
and
\[
 \Sigma_2 = \left\{ (s, t, \sigma, \tau, u, v, \xi, \eta) : \pi_{T\Psi}\hat{R}(u, v, s) \parallel \partial_s\pi_{T\Psi}\hat{R}(u, v, s) \right\} \cap \Lambda^\prime.
\]
\end{proposition}

\begin{proof}
Let $\pi_L : \Lambda^\prime \to T^*\mathcal{Y}\setminus 0$ and $\pi_R
: \Lambda^\prime \to T^*\mathcal{X}\setminus 0$ be the canonical
projections of $T^*(\mathcal{Y}\times\mathcal{X})\setminus 0$ onto its
factors. We will show that $\Lambda^\prime$ is locally of graph type
away from $\Sigma$ by showing that $\nabla\pi_L$ is not of full rank
only on $\Sigma_1$ and $\nabla\pi_R$ is not of full rank only on
$\Sigma_2$.

We have the coordinate representation
\[
 \pi_L(s, \tau, u, v) = \left( s, \frac{2}{c_0}|R(u,v,s)|,
 \frac{2}{c_0}\tau \hat{R}(u, v, s) \cdot \dot{\gamma}(s), \tau
 \right)
\]
from which we can calculate
\[
 \nabla\pi_L =
\left[
\begin{array}{cccc}
1 & 0 & 0 & 0 \\
* & 0 & \frac{2}{c_0}\pi_1\pi_{T\Psi}\hat{R}(u, v, s) & \frac{2}{c_0}\pi_2\pi_{T\Psi}\hat{R}(u, v, s) \\
* & * & \partial_u(\frac{2\tau}{c_0}\hat{R}(u, v, s)\cdot\dot{\gamma}(s)) & \partial_v(\frac{2\tau}{c_0}\hat{R}(u, v, s)\cdot\dot{\gamma}(s)) \\
0 & 1 & 0 & 0
\end{array}
\right]
\]

Here $\pi_1, \pi_2$ are the natural projections onto the first and
second component of $\pi_{T\psi}\hat{R}(u, v, s)$.

As for $\pi_R$, we have that
\[
 \pi_R(s, \tau, u, v) = \left(u, v, \pi_1 \frac{2\tau}{c_0}\pi_{T\Psi}\hat{R}(u, v, s), \pi_2\frac{2\tau}{c_0}\pi_{T\Psi}\hat{R}(u, v, s) \right)
\]
Therefore,
\[
 \nabla \pi_R =
\left[
\begin{array}{cccc}
0 & 0 & \frac{2\tau}{c_0}\pi_1\partial_s\pi_{T\Psi}\hat{R}(u, v, s) & \frac{2\tau}{c_0}\pi_2\partial_s\pi_{T\Psi}\hat{R}(u, v, s) \\
0 & 0 & \frac{2}{c_0}\pi_1\pi_{T\Psi}\hat{R}(u, v, s) & \frac{2}{c_0}\pi_2\pi_{T\Psi}\hat{R}(u, v, s) \\
1 & 0 & * & * \\
0 & 1 & * & *
\end{array}
\right]
\]
This has full rank, provided $(s, \tau, u, v) \not\in \Sigma_2$.
\end{proof}

{\bf Remark}. In a flat topography, linear flight path model of SAR,
the degenerate set reduces to the set of covectors whose base points
lie directly under the flight path. This is related to the well-known
difficulty of imaging below the flight path in SAR.

If instead the flight path is curved, the degenerate set contains for
each base point on the flight path a one-dimensional family of
covectors. These covectors have the property that the line through
them is tangent to the flight path at the base point. That the
canonical relation of $F$ is not of graph type near these covectors
was shown in \cite{nolan04}.

Finally, \cite{nolan03} mentions the difficulty of imaging points on a
non-flat surface that are local minima of the travel-time function
$2|R(u, v, s)|/c_0$. When $(u, v, s)$ is such a point,
$\hat{R}(u,v,s)\nabla\psi(u, v)$ = 0, and so by the above analysis
both left and right projections drop at least one rank.

The degenerate set $\Sigma \subset \Lambda^\prime$ is then the
complement of the largest open subset of $\Lambda^\prime$ which is
locally of graph type. For any point $(p, q) \in \Lambda^\prime$,
there exists a small neighborhood that is not only of graph type, but
is also such that it is the graph of a bijective diffeomorphism. These
neighborhoods will play an essential role in constructing microlocal
parametrices later on.

\section{Mirror points}

The most striking difference between SAR on a flat topography (with
either curved or straight flight path) and SAR on a non-flat
topography is the geometry of sets of mirror points. In the flat case
each point in $T^*\mathcal{Y}\setminus 0$ is associated to at most two
mirror points \cite{nolan04, stefanov13}. In this section, we will
show that such mirror point sets in SAR with non-flat topography may
either not be discrete or may contain more than two mirror points.

\begin{defn}
Fix $p \in T^*\mathcal{Y}\setminus 0$. Then the set of mirror points
associated to $p$ is the set
\[
 M_p = \{ q \in T^*\mathcal{X}\setminus 0 : (q, p) \in \Lambda^\prime
 \} \subset T^*\mathcal{X}\setminus 0.
\]
In other words, if we interpret $\Lambda^\prime$ as a (possibly
multi-valued) function from $T^*\mathcal{X}\setminus 0$ to
$T^*\mathcal{Y}\setminus 0$, $M_p$ is the inverse image of $p$ under
$\Lambda^\prime$.
\end{defn}

In this paper, we consider only mirror points associated to the same
covector along the flight path. However, there is a different, broader
notion of mirror points discussed in \cite{stefanov13}, which are
associated to more than one covector via a billard-like flow. In the
non-flat case, the presence of possibly infinite families of mirror
points makes this sort of global analysis more difficult. This is why
we restrict ourselves to the microlocal notion of mirror points
associated to the same covector.

The following proposition shows that, away from a degenerate set, the
microlocal analysis of mirror points on a non-flat topography is
similar to that on a flat topography, insofar as non-degenerate mirror
points are isolated. However, non-degenerate mirror points may also be
mirror to an infinite family of degenerate mirror points -- a
phenomena which does not occur in the flat case. Further, in some
cases, all singularities may be degenerate.

\begin{proposition}
Fix $p \in T^*\mathcal{Y}\setminus 0$. Let $\Sigma_p$ be the set of
points $q$ such that $(p, q) \in \Sigma$, the degenerate set of the
canonical relation. Let $M_p$ be the set of mirror points associated
to $p$. Then $M_p \setminus \Sigma_p$ consists of isolated covectors.
\end{proposition}

\begin{proof}
By definition $\Sigma_p \subset M_p$. Recall $\Lambda^\prime$ is
locally of graph type near $(p, q) \in \Lambda^\prime$ for all $q \in
M_p \setminus \Sigma_p$. Then, there are suitably small neighborhoods
of each $q$ such that $\Lambda^\prime$ acts as a bijective
diffeomorphism on each neighborhood. Therefore, the non-degenerate
mirror points are isolated.
\end{proof}

{\bf Remark}. If the topography is flat and the flight path is
straight, these non-degenerate mirror points appear to the left and
right of the flight path, causing left-right ambiguity. If the flight
path is curved, the non-degenerate mirror points again come in pairs,
as studied in \cite{stefanov13}.

\begin{figure}[h]
  \begin{center}
    \includegraphics[width=3.5in]{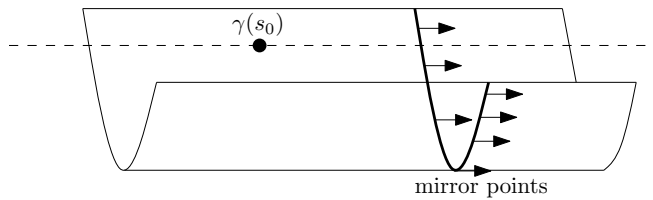}
    \caption{A cylindrical topography with flight path along
    the axis of rotation.}
    \label{cylinder}
  \end{center}
\end{figure}

In Figure \ref{cylinder}, we show an example of a cylindrical
topography, which was also considered in \cite{nolan03}. Let $t >
2r/c_0$, where $r$ is the radius of the cylinder. Associated to every
covector over $(s_0, t_0)$ is an infinite family covectors, with
direction parallel to the axis of the cylinder. Similar examples with
polygonal cross-sections (whose corners have been suitably rounded)
can be constructed to yield any number of mirror points associated to
the same covector.

\section{Cancellation of singularities}

In this section, we consider when signals in the data which would
otherwise be observed in the data cancel each other out, in the manner
of destructive interference. We show that for every pair of
non-degenerate mirror points, there exists an infinite-dimensional
family of ground reflectivity functions $V$ whose image under the
forward operator $F$ is smooth. In other words, the microlocal kernel
of $F$ is infinite-dimensional, which obstructs a stable
reconstruction.

Recall that for a fixed $p \in T^*\mathcal{Y}\setminus 0$, the set of
mirror points associated with $p$ are denoted $M_p$. The first results
of this section concern the case where $M_p$ contains at least two
distinct, isolated mirror points, echoing the situation of SAR with
flat topography.

\subsection{Non-degenerate mirror points}

Fix $p \in T^*\mathcal{Y}\setminus 0$. In this section, we consider
the case when two isolated mirror points are associated to $p$. In
general, the number of isolated mirror points associated to $p$ may be
greater than two. This contrasts with the case of flat topography,
which has at most two mirror points associated to any signal. Our
result in this case generalizes \cite[Theorem 2.1]{stefanov13} to a
different model of SAR taking into account non-flat topography,
arbitrary flight paths, and generic incident waves.

\begin{theorem}
Let $\Gamma \subset T^*\mathcal{Y}\setminus 0$ be a small, open conic
neighborhood of $p$, and suppose there exist isolated mirror points
$q_1, q_2 \in M_p$ associated to $p$. Assume the amplitude of the
forward operator $F$ is nonzero in a neighborhood of
$(\pi_{\mathcal{X}}(q_1), \pi_{\mathcal{Y}}(p)) \times \R\setminus 0$
and $(\pi_{\mathcal{X}}(q_2), \pi_{\mathcal{Y}}(p)) \times \R\setminus
0$. Let $\Gamma_1, \Gamma_2 \subset T^*\mathcal{X}\setminus 0$ be
small, open conic neighborhoods of $q_1, q_2$ respectively whose image
under $\Lambda^\prime$ is $\Gamma$.

Then for every $V_1 \in \cs(\mathcal{X})$ such that $\WF(V_1) \subset
\Gamma_1$, there exists $V_2 \in \cs(\mathcal{X})$ with $\WF(V_2)
\subset \Gamma_2$ related by a Fourier integral operator whose
canonical relation is a diffeomorphism between $\Gamma_1$ and
$\Gamma_2$ such that $\WF(F(V_1 + V_2)) \cap \Gamma = \emptyset$.
\end{theorem}

\begin{proof}
As $q_1, q_2$ are isolated, there is a conic neighborhood of both such
that the restriction of $\Lambda^\prime$ to either is the graph of a
function. After perhaps shrinking both neighborhoods and $\Gamma$, we
can construct two conic neighborhoods $\Gamma_1, \Gamma_2$ of each
covector respectively such that the restriction of $\Lambda^\prime$ to
either set is the graph of a bijective diffeomorphism onto $\Gamma$.

Let $\chi_1, \chi_2 \in C^\infty(T^*\mathcal{X}\setminus 0)$ be two
cut-off functions such that $\chi_i$ is supported on
$\overline{\Gamma_i}$, nonzero on $\Gamma_i$, and homogeneous of
degree zero in the fiber variables. We write the quantization
$\chi_i(u, v, D_u, D_v)$ as $P_i$, which is a pseudodifferential
operator of order zero. Similarly, let $\chi \in
C^\infty(T^*\mathcal{Y}\setminus 0)$ be a cut-off function with the
same properties on $\Gamma$, and write $\chi(s, t, D_s, D_t) = P$.

By construction, $F_i = PFP_i$ is a Fourier integral operator with
canonical relation equal to that of $\Lambda^\prime$ restricted to
$\Gamma_i \times \Gamma$. This canonical relation is the graph of a
diffeomorphism, and moreover that diffeomorphism is
bijective. Further, the amplitude of $F_i$ is nonzero. They are
elliptic in the sense of \cite[Def 25.3.4]{horm85}, so there exist two
microlocal parametrices $F_i^{-1}$.

Given $V_1 \in \cs(\mathcal{X})$ with $\WF(V_1) \subset \Gamma_1$, the
calculus of Fourier integral operators implies that $V_2 =
-F_2^{-1}F_1V_1$ has $\WF(V_2) \subset \Gamma_2$. From here we can
calculate,
\begin{align*}
 F(V_1 + V_2) & = FV_1 + -FF_2^{-1}F_1V_1 \\ 
\intertext{As the $\WF(F_2^{-1}F_1V_1) \subset \Gamma_2$,
  $-FF_2^{-1}F_1V_1 = -F_2F_2^{-1}F_1V_1$ modulo a smoothing operator
  applied to $V_1$, and similarly for $FV_1$.}
 F(V_1 + V_2) &= F_1V_1 - F_2F_2^{-1}F_1V_1 \pmod \cinf \\
  &= F_1V_1 - F_1V_1 \pmod \cinf \\
  &= 0 \pmod \cinf.
\end{align*}
As we have applied a microlocal cut-off supported on $\Gamma$, this
shows the wavefront set of $F(V_1 + V_2)$ does not intersect $\Gamma$.
\end{proof}

Each pair of isolated mirror points therefore induces an
infinite-dimensional family of counter-examples microlocally supported
near that pair. So, at best, any reconstruction of $\WF(V)$ from $FV$
can only proceed up to sets of mirror points. This accounts for the
artifacts seen in many forms of SAR.

In addition, this theorem can be extended to any subset of
non-degenerate mirror points. Given $V_i \in \cs(X), i = 1,\dots, n-1$
microlocally supported near a family $q_i, i = 1,\dots, n$ of
non-degenerate mirror points, the proof above shows that there is
\[
 V_n = -F_n^{-1}\sum_{i=1}^{n-1} F_i V_i
\]
microlocally supported near $q_n$ such that
\[
 \sum_{i=1}^n FV_i \in \cinf.
\]
This yields yet more counter-examples to stable reconstruction.

So far we have only considered mirror points associated to the same
covector above the flight path. As mentioned earlier, there is also a
notion of mirror points associated to multiple covectors, considered
in the flat case by \cite{stefanov13}.

For example, let $q_1, q_2$ be associated to $p_1$, and $q_2, q_3$ be
associated to $p_2$. If Theorem 1 holds for $(q_1, q_2, p_1)$ and
$(q_2, q_3, p_2)$ separately, then one cancel the singularity of
$V_1$, microlocally supported near $q_1$, with $V_3$, microlocally
supported at $q_3$ where,
\[
 V_3 = -F_3^{-1}F_2F_2^{-1}F_1V_1.
\]
A similar argument shows that $F(V_1 + V_3)$ is smooth.

\subsection{Degenerate mirror points}

Finally, we present an example showing that infinite families of
mirror points also permit the cancellation of singularities. It is
important to note that the proof of Theorem 1 depends crucially on the
mirror points being non-degenerate -- otherwise, we cannot pass to a
microlocalization whose canonical relation is of graph type. For this,
we return to the example shown in Figure \ref{cylinder}.

{\bf Example}. Let $\Psi$ be the cylinder given by $\psi(u, v) = (\cos
u, v, 1-\sin u)$, where $(u, v) \in (0, \pi) \times \R$, and
$\gamma(s) = (0, s, 1)$. Let $V = f(u)H(v)$ where $H(v) = \chi_{[0,
    \infty)}$ is the Heaviside function and $f(u) \in \cinf((0,
  \pi))$. Assume $A = 1$ uniformly.

When $A = 1$, the forward operator reduces to the Fourier transform of
a delta function.
\begin{align*}
 FV(s, t) &= \int_{\R\times\Psi} e^{-i\omega\left(t -
   \frac{2}{c_0}\sqrt{(v-s)^2 + 1}\right)} f(u)H(v) \, du\, dv\,
 d\omega \\
  &= \left\langle \delta\left(t - \frac{2}{c_0}\sqrt{(v-s)^2 + 1} \right), f(u)H(v) \right\rangle \\
  &= \left\langle \delta \left(t - \frac{2}{c_0}\sqrt{(v-s)^2 + 1} \right), H(v)\right\rangle \int_0^\pi f(u)\, du.
\end{align*}
The map
\[
 w(v) = t - \frac{2}{c_0}\sqrt{(v-s)^2 + 1}
\]
is two-to-one onto the interval $(-\infty, t - 2/c_0)$. To calculate
the pull-back, we divide the domain of $w$ into two intervals,
$(-\infty, s)$ and $(s, \infty)$. Then there are two inverses $v_-$
and $v_+$ with range on each interval, respectively:
\[
 v_\pm(w) = s \pm \sqrt{\frac{c_0^2}{4}(w - t)^2 - 1}.
\]
In either case, the derivative is non-zero on $(-\infty, t - 2/c_0)$,
and explicitly,
\[
 \left| \frac{dv_\pm}{dw}(w) \right| = \frac{c_0^2}{4}\frac{t - w}{\sqrt{c_0^2(t-w)^2/4 - 1}}.
\]
Using this, we may calculate the pullback of $\delta(w)$ via
$w(v)$. Let 
\[ \alpha(t) = \sqrt{c_0^2t^2/4 - 1}. \]
Then:
\begin{align*}
 \left\langle \delta(w(v)), H(v)\right\rangle &= \left\langle \delta(w), \left| \frac{dv_\pm}{dw}(w)\right|[H(v_-(w)) + H(v_+(w))] \right\rangle \\
   &= \frac{c_0^2}{4} \frac{t}{\sqrt{c_0^2t^2/4 - 1}}[H(s+\alpha(t))+H(s-\alpha(t))] 
\end{align*}
So, the forward operator reduces to
\[
 FV(s, t) = \frac{c_0^2}{4}\frac{t[H(s+\alpha(t))+H(s-\alpha(t))]}{\sqrt{c_0^2t^2/4 - 1}} \int_0^\pi f(u)\, du,
\]
which vanishes whenever the integral of $f$ vanishes. There is a
subspace of $\cinf((0, \pi))$ for which this is true with infinite
dimension. Since
\[ 
 \WF(f(u)H(v)) = (\{v=0\} \times \{\xi = 0\}) \setminus 0
\] 
and $\WF(FV(f(u)H(v))) = \emptyset$, this shows that singularities on
degenerate mirror points may also cancel.

\bibliographystyle{abbrv}
\bibliography{master}

\end{document}

%% file: standard.tex
\usepackage{amsmath, amssymb, amsthm}
\usepackage{graphicx}
\usepackage{epstopdf}

\newtheorem*{defn}{Definition}

\newtheorem*{notation}{Notation}
\newtheorem{theorem}{Theorem}

\newtheorem{proposition}{Proposition}

\newtheorem*{lemma*}{Lemma}
\newtheorem*{thrm*}{Theorem}

\newcommand{\cinf}{C^\infty}

\newcommand{\cs}{\mathcal{E}^\prime}

\newcommand{\WF}{\mathop{\mathrm{WF}}}

\newcommand{\R}{\mathbb{R}}

%% file: sar-final.bbl
\begin{thebibliography}{10}

\bibitem{agranovsky07}
M.~Agranovsky and P.~Kuchment.
\newblock Uniqueness of reconstruction and an inversion procedure for
  thermoacoustic and photoacoustic tomography with variable sound speed.
\newblock {\em Inverse Problems}, 23:2089, 2007.

\bibitem{albanese13}
R.~Albanese and R.~Medina.
\newblock Materials identification synthetic aperture radar: progress toward a
  realized capability.
\newblock {\em Inverse Problems}, 29:054001, 2013.

\bibitem{ambartsoumian12}
G.~Ambartsoumian, R.~Felea, V.~Krishnan, C.~Nolan, and E.~T. Quinto.
\newblock A class of singular {F}ourier integral operators in synthetic
  aperture radar imaging.
\newblock {\em J. Funct. Anal.}, 264:246--269, 2012.

\bibitem{ambartsoumian05}
G.~Ambartsoumian and P.~Kuchment.
\newblock On the injectivity of the circular radon transform.
\newblock {\em Inverse Problems}, 21:473, 2005.

\bibitem{ambartsoumian06}
G.~Ambartsoumian and P.~Kuchment.
\newblock A range description for the planar circular {R}adon transform.
\newblock {\em SIAM J. Math. Anal.}, 38:681--692, 2006.

\bibitem{cheney03}
M.~Cheney and B.~Borden.
\newblock Microlocal structure of inverse synthetic aperture radar data.
\newblock {\em Inverse Problems}, 19:173, 2003.

\bibitem{cheney09}
M.~Cheney and B.~Borden.
\newblock {\em Fundamentals of Radar Imaging}.
\newblock SIAM, 2009.

\bibitem{cheney04}
M.~Cheney and C.~Nolan.
\newblock Synthetic-aperture imaging through a dispersive layer.
\newblock {\em Inverse Problems}, 20:507--532, 2004.

\bibitem{cloude04}
S.~R. Cloude, D.~G. Corr, and M.~L. Williams.
\newblock Target detection beneath foliage using polarimetric synthetic
  aperture radar interferometry.
\newblock {\em Waves in Random Media}, 14:S393--S414, 2004.

\bibitem{coker07}
J.~D. Coker and A.~H. Tewfik.
\newblock Multistatic {SAR} image reconstruction based on an
  elliptical-geometry radon transform.
\newblock In {\em International Waveform Diversity and Design Conference},
  pages 204--208, June 2007.

\bibitem{collard05}
F.~Collard, F.~Ardhuin, and B.~Chapron.
\newblock Extraction of coastal ocean wave fields from {SAR} images.
\newblock {\em IEEE Journal of Oceanic Engineering}, 30:526--533, 2005.

\bibitem{duistermaat96}
J.~J. Duistermaat.
\newblock {\em Fourier integral operators}.
\newblock Birkh\"auser, 1996.

\bibitem{finch04}
D.~Finch, S.~K. Patch, and Rakesh.
\newblock Determining a function from its mean values over a family of spheres.
\newblock {\em SIAM J. Math. Anal.}, 35(5):1213--1240, 2004.

\bibitem{haltmeier07}
M.~Haltmeier, O.~Scherzer, P.~Burgholzer, R.~Nuster, G.~Paltauf, and
  N.~Bellomo.
\newblock Thermoacoustic tomography and the circular {R}adon transform: Exact
  inversion formula.
\newblock {\em Math. Models Methods Appl. Sci.}, 17:635--655, 2007.

\bibitem{hellsten87}
H.~Hellsten and L.~Andersson.
\newblock An inverse method for the processing of synthetic aperture radar
  data.
\newblock {\em Inverse Problems}, 3:111--124, 1987.

\bibitem{horm85}
L.~H\"ormander.
\newblock {\em The Analysis of Linear Partial Differential Operators IV:
  Fourier Integral Operators}.
\newblock Springer, 1985.

\bibitem{krishnan11}
V.~Krishnan and E.~Quinto.
\newblock Microlocal aspects of common offset synthetic aperture radar imaging.
\newblock {\em Inverse Problems and Imaging}, 5:659--674, 2011.

\bibitem{kuchment08}
P.~Kuchment and L.~Kunyansky.
\newblock Mathematics of thermoacoustic tomography.
\newblock {\em Euro. J Appl. Math.}, 19:191--224, 2008.

\bibitem{nolan03}
C.~Nolan and M.~Cheney.
\newblock Synthetic aperture inversion for arbitrary flight paths and non-flat
  topography.
\newblock {\em IEEE Trans. Image Process.}, 12:1035--1044, 2003.

\bibitem{nolan04}
C.~Nolan and M.~Cheney.
\newblock Microlocal analysis of synthetic aperture radar imaging.
\newblock {\em J. Fourier Anal. Appl.}, 10:133--148, 2004.

\bibitem{stefanov13}
P.~Stefanov and G.~Uhlmann.
\newblock Is a curved flight path in {SAR} better than a straight one?
\newblock SIAM J. Appl. Math. 2013, to appear.

\bibitem{swift79}
C.~Swift and L.~R. Wilson.
\newblock Synthetic aperture radar imaging of moving ocean waves.
\newblock {\em IEEE Trans. Antennas and Propagation}, 27:725--729, 1979.

\bibitem{varslot10}
T.~Varslot, J.~H. Morales, and M.~Cheney.
\newblock Synthetic-aperture radar imaging through dispersive media.
\newblock {\em Inverse Problems}, 26:025008, 2010.

\end{thebibliography}
